\documentclass{amsart}

 \usepackage{amssymb}
\usepackage{amsthm}

\newtheorem{theorem}{Theorem}[section]
\newtheorem{lemma}{Lemma}[section]
\newtheorem{remark}{Remark}[section]
\numberwithin{equation}{section}
\newtheorem*{theorem*}{Theorem}

\begin{document}

\subjclass[2010]{Primary 46E35. Secondary 26A33, 26D10, 46E30.}

\title[The Bourgain-Br\'ezis-Mironescu formula in arbitrary bounded domains]{The Bourgain-Br\'ezis-Mironescu formula in arbitrary bounded domains}

\author{Irene Drelichman}
\address{IMAS (UBA-CONICET), Facultad de Ciencias Exactas y Naturales, Universidad de Buenos Aires, Ciudad Universitaria, 1428 Buenos Aires, Argentina}
\email{irene@drelichman.com}

\author{Ricardo G. Dur\'an}
\address{IMAS (UBA-CONICET) and Departamento de Matem\'atica, Facultad de Ciencias Exactas y Naturales, Universidad de Buenos Aires, Ciudad Universitaria, 1428 Buenos Aires, Argentina}
\email{rduran@dm.uba.ar}

\thanks{Supported by FONCYT under grants  PICT-2018-03017 and PICT-2018-00583, and by Universidad de Buenos Aires under grant 20020160100144BA}

\begin{abstract}
We obtain a  Bourgain-Br\'ezis-Mironescu formula on the limit behaviour of a modified fractional Sobolev seminorm when $s\nearrow 1$, which is valid in arbitrary bounded domains. In the case of extension domains, we recover the classical result.
\end{abstract}

\keywords{Fractional Sobolev spaces, Gagliardo seminorm, irregular domains.}

\maketitle

\section{Introduction}

The celebrated Bourgain-Br\'ezis-Mironescu formula \cite{BBM} (see also \cite{MS,M}) states that, for any extension domain $\Omega \subset \mathbb{R}^n$, $1<p<\infty$, and $f\in L^p(\Omega)$,
\begin{equation}
\label{BBM}
\lim_{s\nearrow 1} (1-s)  \int_\Omega \int_\Omega \frac{|f(x)-f(y)|^p}{|x-y|^{n+sp}} \, dx \, dy = K_{n,p}  \int_\Omega |\nabla f(x)|^p \, dx
\end{equation}
where $K_{n,p}= \frac{1}{p}\int_{\mathbb{S}^{n-1}} |\sigma_n|^p \, d\sigma$, with the convention that  $\|\nabla f\|_p =\infty$ if $f\not\in W^{1,p}(\Omega)$.  

The aim of this note is to obtain an analogous result in bounded irregular domains. As it is usual, for $0<s<1$,  we will denote
$$
W^{s,p}(\Omega)= \left\{ f \in L^p(\Omega) : \int_{\Omega} \int_\Omega \frac{|f(x)-f(y)|^p}{|x-y|^{n+sp}} \, dx \, dy   <\infty  \right\}.
$$
One can easily check that, for arbitrary $\Omega$, it may occur that $W^{1,p}(\Omega)\not\subset W^{s,p}(\Omega)$ (see, e.g. \cite[Example 2.1]{DD-interpolacion}). Therefore, it is not reasonable to expect \eqref{BBM} to hold verbatim in all domains  (see \cite[Remark 5]{B}).  

The author of \cite{B} suggested a possible solution would be to consider $\delta(x,y)$ the geodesic distance in $\Omega$ and asked whether, for $f\in L^p(\Omega)$ and $\Omega \subset \mathbb{R}^n$ a bounded domain, 
\begin{equation}
\label{limsup}
\lim\sup_{s\nearrow 1} (1-s)\int_\Omega \int_\Omega \frac{|f(x)-f(y)|^p}{\delta(x,y)^{n+sp}} \, dx \, dy < \infty
\end{equation}
would guarantee that $f \in W^{1,p}(\Omega)$ and, if so, whether
$$
\lim_{s\nearrow 1} (1-s)\int_\Omega \int_\Omega \frac{|f(x)-f(y)|^p}{\delta(x,y)^{n+sp}} \, dx \, dy =  K_{n,p}  \int_\Omega |\nabla f(x)|^p \, dx
$$
 (see \cite[Open problem 1]{B}). For $1<p<\infty$, the first question was positively answered in \cite[Corollary 1.7]{LS}. To the best of our knowledge, the second one is still unanswered, but the authors of \cite{LS}  found  a different limit that equals $K_{n,p} \|\nabla f\|_p^p$ under assumption \eqref{limsup} in arbitrary domains (see \cite[Theorem 1.5]{LS} and \cite{LS2} for a precise statement).  

As announced, we also tackle the problem of finding a suitable replacement of \eqref{BBM} for arbitrary bounded domains, but propose a different solution. Actually, we also give a positive answer to the first question in  \cite[Open problem 1]{B} and show that, in addition, one may replace hypothesis \eqref{limsup} by a weaker one, namely, changing $\lim\sup$ by $\lim\inf$. 

Before we state our  theorem, let us recall that, as shown in \cite{M},  there is a deep connection between the limit \eqref{BBM} and the fact that, for $\Omega$ an extension domain or the whole space $\mathbb{R}^n$, $W^{s,p}(\Omega)$ is the real interpolation space between $L^p(\Omega)$ and  $W^{1,p}(\Omega)$, i.e. $(L^p(\Omega), W^{1,p}(\Omega))_{s,p}=W^{s,p}(\Omega)$. This is clearly not the case if the inclusion $W^{1,p}(\Omega)\subset W^{s,p}(\Omega)$ fails. 

The  characterization of the interpolation space $(L^p(\Omega), W^{1,p}(\Omega))_{s,p}$ for arbitrary $\Omega$ is a difficult open problem, but in  \cite{DD-interpolacion} the authors proved that, for a certain class of irregular domains, one has  $(L^p(\Omega), W^{1,p}(\Omega))_{s,p}=\widetilde{W}^{s,p}(\Omega)$, with
$$
\widetilde{W}^{s,p} (\Omega)= \left\{ f \in L^p(\Omega) : \int_\Omega \int_{|x-y|<\tau d(x)} \frac{|f(x)-f(y)|^p}{|x-y|^{n+sp}} \, dx \, dy   <\infty  \right\},
$$
where $\tau\in (0,1)$ is a fixed parameter and $d(x)=\mbox{dist}(x, \partial\Omega)$. 

This fractional space had been previously introduced in \cite{HV} and  plays an important role in the study of fractional Poincar\'e and Sobolev-Poincar\'e inequalities in irregular domains (see \cite{DD-fracpoincare, DIV, HV,  MP}). More importantly, it is known that it coincides with  $W^{s,p}(\Omega)$ (with equivalence of norms) when $\Omega$ is a Lipschitz domain \cite[Proposition 5]{D} or, more generally, a uniform domain \cite[Corollary 4.5]{PS}. 

Moreover, a careful reading of the proof in \cite{BBM} shows that, when $\Omega$  is an extension domain and $f$ a smooth function, one has
\begin{equation}
\label{lim0}
\lim_{s\nearrow 1} (1-s) \int_\Omega \int_{|x-y|>\tau d(x)} \frac{|f(y)-f(x)|^p}{|x-y|^{n+sp}} \, dy \, dx = 0
\end{equation}
(see  \cite{BBM} after equation (5)). However, adding this zero term allows one to recover the Gagliardo seminorm, which is  symmetric  in $x$ and $y$, a fact  that is essential for the rest of the proof in \cite{BBM}. Needless to say, \eqref{lim0} fails badly in irregular domains (one may use again \cite[Example 2.1]{DD-interpolacion}).

The above results strongly support the idea that, replacing the $W^{s,p}(\Omega)$ seminorm by the $\widetilde{W}^{s,p}(\Omega)$ seminorm in \eqref{BBM}, one should  obtain a formula valid in any bounded domain. 
In fact, we prove:
\begin{theorem}
\label{teo}
Let $\Omega \subset \mathbb{R}^n$ be a bounded domain, $d(x)=\mbox{dist}(x, \partial\Omega)$, and $\tau \in (0,1)$. If $1<p<\infty$ and $f\in L^p(\Omega)$,
\begin{equation}
\label{BBM-rest}
\lim_{s\nearrow 1} (1-s) \int_\Omega \int_{|x-y|<\tau d(x)} \frac{|f(y)-f(x)|^p}{|x-y|^{n+sp}} \, dy \, dx = K_{n,p} \int_\Omega |\nabla f (x)|^p \, dx,
\end{equation}
where $K_{n,p}= \frac{1}{p} \int_{\mathbb{S}^{n-1}} |\sigma_n|^p \, d\sigma$, with the convention that  $\|\nabla f\|_p =\infty$ if $f\not\in W^{1,p}(\Omega)$.  
\end{theorem}

\begin{remark}
When  $|x-y|<\tau d(x)$ for $\tau \in (0,1)$, it is clear that the geometric distance equals the euclidean distance,  that is, $\delta(x,y)=|x-y|$. Therefore, Theorem  \ref{teo} and the obvious inequality
\begin{align*}
\lim\inf_{s\nearrow 1} (1-s)&\int_\Omega \int_{|x-y|<\tau d(x)} \frac{|f(x)-f(y)|^p}{\delta(x,y)^{n+sp}} \, dx \, dy \\
&\le \lim\sup_{s\nearrow 1} (1-s)\int_\Omega \int_\Omega \frac{|f(x)-f(y)|^p}{\delta(x,y)^{n+sp}} \, dx \, dy 
\end{align*}
show that the answer to the first part of \cite[Open problem 1]{B} is positive. 

Also, in view of Theorem \ref{teo}, the second part of the problem is true if and only if
$$
\lim_{s\nearrow 1} (1-s) \int_\Omega \int_{|x-y|>\tau d(x)} \frac{|f(y)-f(x)|^p}{\delta(x,y)^{n+sp}} \, dy \, dx =0.
$$
\end{remark}

The fact that the  $\widetilde W^{s,p}(\Omega)$ seminorm, unlike those in \eqref{BBM} or \eqref{limsup},  is not symmetric  is one of the main difficulties of our proof, another one being the fact that, in irregular domains, we cannot count on the density of smooth functions up to the boundary, so that some bounds (to apply dominated convergence in Step 3 below, for instance) will require more careful arguments than in the original proof. We recall some tools that we will need to handle these difficulties in the next section, and devote the last section of this paper to the proof itself.

\section{Preliminary results}

As mentioned before, for arbitrary $\Omega$ it may happen that $W^{1,p}(\Omega)\not\subset W^{s,p}(\Omega)$. However, one always has the inclusion  $W^{1,p}(\Omega)\subset \widetilde W^{s,p}(\Omega)$. To prove it, we will use the following lemma, which is a special case of \cite[Lemma 7]{C}:
\begin{lemma}
\label{lema calderon}
For $1<p<\infty$, $h\in L^p(\mathbb{R}^n)$, $h\ge 0$,  and $\nu$ a unit vector, let
$$
h_1(x,\nu) = \sup_{t>0} \frac{1}{t} \int_0^t h(x+\nu s) \, ds 
$$ 
and
$$
h^*(x) = \left(\int_{\mathbb{S}^{n-1}} h_1(x,\nu)^p \, d\sigma_\nu \right)^\frac{1}{p}.
$$
Then, $h^* \in L^p$ and $\|h^*\|_p\le C \|h\|_p$.
\end{lemma}

Now we are ready to prove our claim:
\begin{lemma}
\label{lema1}
Let $\Omega \subset \mathbb{R}^n$ be a bounded domain, 
$$
F_s(x)= \int_{|h|<\tau d(x)} \frac{|f(x+h)-f(x)|^p}{|h|^{n+sp} } \, dh,
$$
and 
$$
|f|_{\widetilde W^{s,p}(\Omega)} = \left( \int_\Omega F_s(x) \, dx\right)^\frac{1}{p}.
$$

Then, for $1<p<\infty$, we have the pointwise bound
\begin{equation*}
\label{pointwise}
(1-s) F_s(x) \le C |\nabla f|^*(x)^p.
\end{equation*}

In particular, 
$$(1-s) |f|_{\widetilde W^{s,p}(\Omega)}^p \le C \| \nabla f\|_{p}^p,$$
which means $W^{1,p}(\Omega) \subset \widetilde{W}^{s,p}(\Omega)$.
\end{lemma}
\begin{proof}
By the Meyers-Serrin theorem \cite{MeSe}, we may assume that $f$ is smooth. Then, if $\sigma=\frac{h}{|h|}$ and $r=|h|$,
$$
f(x+h)-f(x) =\int_0^1 \nabla f(x+th) \cdot h \, dt = \int_0^r \nabla f(x+s\sigma) \cdot \sigma \, ds 
$$
which implies that, extending $|\nabla f|$ by zero, 
\begin{align*}
(1-s) F_s(x) &= (1-s) \int_{|h|<\tau d(x)} \left| \int_0^r \nabla f (x+s\sigma) \cdot \sigma \, ds \right|^p \frac{1}{|h|^{n+sp}} \, dh\\
& \le (1-s) \int_0^{\tau d(x)} \int_{\mathbb{S}^{n-1}} \left( \int_0^r  |\nabla f (x+s\sigma)| \, ds  \right)^p  \, d\sigma \, \frac{1}{r^{sp+1}} \, dr\\
& = (1-s) \int_0^{\tau d(x)} \int_{\mathbb{S}^{n-1}} \left( \frac{1}{r} \int_0^r  |\nabla f (x+s\sigma)| \, ds  \right)^p  \, d\sigma \, r^{p-sp-1} \, dr\\
& \le (1-s) \int_0^{\tau d(x)} |\nabla f|^* (x)^p \, r^{p-sp-1} \, dr\\
&= \frac{(\tau d(x))^{p(1-s)}}{p} |\nabla f|^*(x)^p.
 \end{align*}
 
Finally, by Lemma \ref{lema calderon}, we obtain
$$
(1-s) |f|_{\widetilde W^{s,p}(\Omega)}^p \le C \|\nabla f\|_p^p.
$$
\end{proof}

\bigskip
\section{Proof of Theorem \ref{teo}}

As in \cite{BBM}, it suffices to show that \eqref{BBM-rest} holds for $f\in W^{1,p}(\Omega)$, and that, if $f\in L^p(\Omega)$ and
$$
\widetilde{A} =\liminf_{s \nearrow 1} (1-s) \int_\Omega \int_{|h|<\tau d(x)} \frac{|f(x+h)-f(x)|^p}{|h|^{n+sp}} \, dh \, dx < \infty, 
$$
then $f \in W^{1,p}(\Omega)$. We divide the proof into several steps. As mentioned in the introduction, the first two steps mimic closely those in \cite{BBM}, and it is the rest of the proof that requires different arguments. However, we have chosen to include full proofs of all the steps to make this presentation self-contained. 

\medskip

{\bf STEP 1:} We show that, for $f\in C^2(\Omega)$,
$$
\lim_{s \nearrow 1} (1-s)  \int_{|x-y|< \tau d(x)}  \left(\left| \nabla f(x) \cdot \frac{x-y}{|x-y|^s}\right|\right)^p \frac{1}{|x-y|^n} \, dy
 = K_{n,p}  |\nabla f(x)|^p 
$$
with $K_{n,p}= \frac{1}{p} \int_{\mathbb{S}^{n-1}} |\sigma_n|^p \, d\sigma$.
\begin{proof}
Since $x$  is fixed, we may assume, by rotation, that the vector   $\nabla f(x)$ has only its $n$-th coordinate different  from zero. Then,
\begin{align*}
\int_{|x-y|< \tau d(x)}  \left(\left| \nabla f(x) \cdot \frac{x-y}{|x-y|^s}\right|\right)^p \frac{1}{|x-y|^n}\, dy &= \int_0^{\tau d(x)} \int_{\mathbb{S}^{n-1}} |\nabla f(x) \cdot \sigma|^p r^{(1-s)p} \, d\sigma \frac{dr}{r}\\
&= \int_0^{\tau d(x)} \int_{\mathbb{S}^{n-1}} |\nabla f(x)|^p  |\sigma_n|^p r^{(1-s)p} \, d\sigma \frac{dr}{r}\\
&=  K_{n,p} \,  |\nabla f(x)|^p \, \frac{(\tau d(x))^{(1-s)p}}{(1-s)}
\end{align*}
and, therefore
$$
\lim_{s \nearrow 1} (1-s)  \int_{|x-y|< \tau d(x)}  \left(\left| \nabla f(x) \cdot \frac{x-y}{|x-y|^s}\right|\right)^p  \frac{1}{|x-y|^n} \, dy
 = K_{n,p} |\nabla f(x)|^p 
$$
\end{proof}

\medskip

{\bf STEP 2:}  We show that, for $f\in C^2(\Omega)$,
$$
\lim_{s \nearrow 1} (1-s) \int_{|x-y|< \tau d(x)} \frac{|f(x)-f(y)|^p}{|x-y|^{sp}} \frac{1}{|x-y|^n} \, dy = K_{n,p} |\nabla f(x)|^p.
$$
\begin{proof}
By the previous step, it suffices to show that 
\begin{align*}
\lim_{s \nearrow 1} (1-s) & \int_{|x-y|< \tau d(x)} \frac{|f(x)-f(y)|^p}{|x-y|^{sp}} \frac{1}{|x-y|^n} \, dy \\
&=  \lim_{s \nearrow 1} (1-s)  \int_{|x-y|< \tau d(x)}  \left(\left| \nabla f(x) \cdot \frac{x-y}{|x-y|^s}\right|\right)^p \frac{1}{|x-y|^n} \, dy.
\end{align*}

Now, since the limit of the right-hand side exists, it suffices to show that
\begin{equation*}
\lim_{s \nearrow 1} (1-s)  \int_{|x-y|< \tau d(x)} \left| \left(\frac{|f(x)-f(y)|}{|x-y|^s}\right)^p - \left( \left| \nabla f(x) \cdot \frac{x-y}{|x-y|^s}\right| \right)^p \right| \frac{1}{|x-y|^n} \, dy =0.
\end{equation*}

For fixed $x$ and all $y$ such that $|x-y|<\tau d(x)$, since taking $p$ powers is locally Lipschitz, we have that
\begin{align*}
\left| \left(\frac{|f(x)-f(y)|}{|x-y|^s}\right)^p - \left( \left| \nabla f(x) \cdot \frac{x-y}{|x-y|^s}\right| \right)^p \right| & \le C \frac{|f(x)-f(y)-\nabla f(x)\cdot (x-y)|}{|x-y|^s}\\
&\le C |x-y|^{2-s}
\end{align*}
(where the constant depends on $x$).

Combining this bound with
$$
\int_{|x-y|<\tau d(x)} |x-y|^{2-s-n} \, dy = \omega_n \int_0^{\tau d(x)} r^{1-s} \, dr = \frac{\omega_n}{2-s} (\tau d(x))^{2-s}
$$
and
$$
\lim_{s\nearrow 1} (1-s) \int_{|x-y|<\tau d(x)} |x-y|^{2-s-n} \, dy = 0, 
$$
we obtain the desired equality. 
\end{proof}

\medskip 

{\bf STEP 3:} We show that, for $f\in C^2(\Omega) \cap W^{1,p}(\Omega)$ 
$$
\lim_{s\nearrow 1} (1-s) |f|_{\widetilde W^{s,p}}^p = K_{n,p} \int_\Omega |\nabla f(x)|^p \, dx.
$$
\begin{proof}
By dominated convergence, it suffices to show that there exists an integrable $g$ such that $(1-s)|F_s(x)|\le g(x)$ with
$$
F_s(x) =\int_{|x-y|< \tau d(x)} \frac{|f(x)-f(y)|^p}{|x-y|^{n+sp}} \, dy,
$$
but this follows immediately by Lemma \ref{lema1}, taking $g=(|\nabla f|^*)^p$. 
\end{proof}

\medskip

{\bf STEP 4:}  We show that, for $f\in W^{1,p}(\Omega)$, 

\begin{equation}
\label{lim-C2}
\lim_{s\nearrow 1} (1-s) |f|_{\widetilde W^{s,p}}^p = K_{n,p} \int_\Omega |\nabla f(x)|^p \, dx.
\end{equation}

\begin{proof}

This follows by density. Namely, we consider $(f_k)_{k\in \mathbb{N}} \subset C^2(\Omega)\cap W^{1,p}(\Omega)$ such that $f_k\to f$ in $W^{1,p}(\Omega)$. Then, 
\begin{align*}
 (1-s) |f|_{\widetilde W^{s,p}}^p - K_{n,p}  \|\nabla f\|_p^p &\le  (1-s) \Big| |f|_{\widetilde W^{s,p}}^p -  |f_k|_{\widetilde W^{s,p}}^p \Big|  \\
 &+ \Big| (1-s) |f_k|_{\widetilde W^{s,p}}^p - K_{n,p}  \|\nabla f_k\|_p^p \Big| \\
 &+  K_{n,p} \Big|  \|\nabla f_k\|_p^p - \|\nabla f\|_p^p \Big|.
 \end{align*}

But, since $f_k \to f$ in $W^{1,p}(\Omega)$,  by Lemma \ref{lema1} the first and last terms in the right hand-side of the previous inequality converge to zero. Finally, by Step 3, 
$$
\lim_{s\nearrow 1} (1-s) |f_k|_{\widetilde W^{s,p}}^p = K_{n,p}  \|\nabla f_k\|_p^p,
$$
and we obtain \eqref{lim-C2}.
\end{proof}

\medskip 

{\bf STEP 5:} We show that, if  $f\in L^p(\Omega)$ and 
$$
\widetilde{A} =\liminf_{s \nearrow 1} (1-s) \int_\Omega \int_{|h|<\tau d(x)} \frac{|f(x+h)-f(x)|^p}{|h|^{n+sp}} \, dh \, dx < \infty, 
$$
then $f \in W^{1,p}(\Omega)$, which completes the proof of the theorem.

\begin{proof}
Consider an exhaustion $\{\Omega_j\}_{j \in \mathbb{N}}$ of smooth bounded domains such that 
$$
\overline\Omega_j\subset \Omega_{j+1} \, \mbox{ and } \,  \bigcup_{j=1}^\infty \Omega_j =\Omega.
$$
Then, there exist a decreasing sequence of positive numbers  $\{\alpha_j\}_{j\in \mathbb{N}}$ such that, for all $x \in \Omega_j$, $d(x)=dist(x, \partial \Omega)\ge \alpha_j >0$. 

We want to  apply the Bourgain-Br\'ezis-Mironescu formula \eqref{BBM} to the sets $\Omega_j$ and prove that the limit is a finite quantity independent of $j$. To this end, we write
\begin{align*}
\int_{\Omega_j} \int_{\Omega_j} \frac{|f(y)-f(x)|^p}{|x-y|^{n+sp} } \, dy \, dx &= \int_{\Omega_j} \int_{\{y\in \Omega_j : |x-y|<\tau d(x)\}}  \frac{|f(y)-f(x)|^p}{|x-y|^{n+sp} } \, dy \, dx\\
& + \int_{\Omega_j} \int_{\{y\in \Omega_j : |x-y|\ge\tau d(x)\}}  \frac{|f(y)-f(x)|^p}{|x-y|^{n+sp} } \, dy \, dx\\
&= {\Large{\textcircled{\small I}}} +{ \Large{\textcircled{\small II}} }
\end{align*}

Clearly, 
\begin{equation*}
{\Large{\textcircled{\small I}}} \le \int_{\Omega} \int_{|x-y|<\tau d(x)}  \frac{|f(y)-f(x)|^p}{|x-y|^{n+sp} } \, dy \, dx
\end{equation*}
so that, by hypothesis, 
$$
\liminf_{s \nearrow 1} (1-s) \,  {\Large{\textcircled{\small I}}} \le \widetilde{A}. 
$$

To bound {\Large{\textcircled{\small II}}}, we use that $d(x)\ge \alpha_j>0$ for all $x\in \Omega_j$ and obtain 
\begin{align*}
{\Large{\textcircled{\small II}}} &\le \int_{\Omega_j} \int_{\{y\in \Omega_j : |x-y|\ge \tau \alpha_j\}}  \frac{|f(y)-f(x)|^p}{|x-y|^{n+sp} } \, dy \, dx\\
&\le  \frac{1}{(\tau \alpha_j)^{n+sp}}  \int_{\Omega_j} \int_{\Omega_j} |f(y)-f(x)|^p \, dy \, dx\\
&\le \frac{2^{p+1}  |\Omega_j|}{(\tau \alpha_j)^{n+sp}} \|f\|_p^p
\end{align*}
so that, 
$$
\liminf_{s \nearrow 1} (1-s) \,  {\Large{\textcircled{\small II}}} =0. 
$$

Therefore, since  $\liminf_{s \nearrow 1} (1-s)  |f|_{W^{s,p}(\Omega_j)}\le \widetilde A$, by \eqref{BBM} we have that $ f \in W^{1,p}(\Omega_j)$ for each $j$ and that the seminorms are uniformly bounded. Hence,   $ f \in W^{1,p}(\Omega)$, as we wanted to see. 

\end{proof}


\begin{thebibliography}{}

\bibitem{BBM} Bourgain, J.; Brezis, H.; Mironescu, P. \emph{Another look at Sobolev spaces}. Optimal control and partial differential equations, 439--455, IOS, Amsterdam, 2001. 

\bibitem{B} Brezis, H. \emph{How to recognize constant functions. A connection with Sobolev spaces}. Russian Math. Surveys 57 (2002), no. 4, 693--708 

\bibitem{C} Calder\'on, A. P. \emph{Estimates for singular integral operators in terms of maximal functions}. Studia Math. 44 (1972), 563--582.

\bibitem{DD-fracpoincare} Drelichman, I.; Dur\'an, R. G. \emph{Improved Poincar\'e inequalities in fractional Sobolev spaces}. Ann. Acad. Sci. Fenn. Math. 43 (2018), no. 2, 885--903.

\bibitem{DD-interpolacion} Drelichman, I.; Dur\'an, R. G. \emph{On the interpolation space $(L^p(\Omega),W^{1,p}(\Omega))_{s,p}$ in non-smooth domains}. J. Math. Anal. Appl. 470 (2019), no. 1, 91--101.


\bibitem{D} Dyda, B. \emph{On comparability of integral forms}. J. Math. Anal. Appl. 318 (2006), no. 2, 564--577. 

\bibitem{DIV} Dyda, B.; Ihnatsyeva, L.; V\"ah\"akangas, A. V. \emph{On improved fractional Sobolev-Poincar\'e inequalities}. Ark. Mat. 54 (2016), no. 2, 437--454.


\bibitem{HV} Hurri-Syrj\"anen, R.; V\"ah\"akangas, A. V. \emph{On fractional Poincaré inequalities}. J. Anal. Math. 120 (2013), 85--104. 

\bibitem{LS} Leoni, G.; Spector, D. \emph{Characterization of Sobolev and $BV$ spaces}. J. Funct. Anal. 261 (2011), no. 10, 2926--2958.

\bibitem{LS2} Leoni, G.; Spector, D. Corrigendum to ``Characterization of Sobolev and $BV$ spaces'' [J. Funct. Anal. 261 (10) (2011) 2926–2958]. J. Funct. Anal. 266 (2014), no. 2, 1106--1114. 

\bibitem{MP} Mart\'inez-Perales, J.C. \emph{A note on generalized Poincar\'e-type inequalities with applications to weighted improved Poincar\'e-type inequalities}. 
 Preprint arXiv:1907.12435, 2019.

\bibitem{MS} Maz'ya, V.; Shaposhnikova, T. \emph{On the Bourgain, Brezis, and Mironescu theorem concerning limiting embeddings of fractional Sobolev spaces}. J. Funct. Anal. 195 (2002), no. 2, 230--238. 


\bibitem{MeSe} Meyers, N. G.; Serrin, J. \emph{$H=W$}. Proc. Nat. Acad. Sci. U.S.A. 51 (1964), 1055--1056.

\bibitem{M} Milman, M. \emph{Notes on limits of Sobolev spaces and the continuity of interpolation scales}. Trans. Amer. Math. Soc. 357 (2005), no. 9, 3425--3442.

\bibitem{PS} Prats, M.; Saksman, E. \emph{A ${\rm T}(1)$ theorem for fractional Sobolev spaces on domains}. J. Geom. Anal. 27 (2017), no. 3, 2490--2538.


 \end{thebibliography}
\end{document}